\documentclass[12pt]{amsart}

\newtheorem{lemma}{Lemma}
\newtheorem{thm}[lemma]{Theorem}

\theoremstyle{definition}

\theoremstyle{remark}
\newtheorem{rmk}[lemma]{Remark}

\numberwithin{equation}{section} \numberwithin{table}{section}

\title[Complex dimensions for IFS with overlaps]
{Complex dimensions for IFS with overlaps}

\author{Nikita Sidorov}
\address{University of Great Northern Tower, 1 Watson Street, Manchester M3 4EE, United Kingdom}
\date{\today}
\dedicatory{To the memory of Pafnutiy L'vovich Chebysh\"ev and Sergei Natanovich Bernstein}

\begin{document}

\begin{abstract}
The notion of {\sl complex dimension} of a one-dimensional
Cantor set $C=\bigcap_{n=1}^\infty  C_n$ dates back
decades \cite{L}. It is defined as the set of poles of the meromorphic
$\zeta$-function $\zeta(s)=\sum_{n=1}^{\infty}d_j^s$, where
$\Re s>0$, and $d_j$ is the length of the $j$th interval in $C_n$.

Following the trend, I switch from sets to measures, which will allow me
to generalize the construction to iterated function schemes that do not
necessarily satisfy the Open Set Condition.
\end{abstract}

\maketitle

Let $\{f_1,\dots,f_m\}$ be the iterated function scheme
on $\mathbb R$ with $f_i(x)=\rho x + (1-\rho)a_i$, where $\rho\in(0,1), a_i\in\mathbb N,
a_1=0<a_2<\dots<a_m$.
Let $p_1,\dots,p_m\in(0,1)$ with $p_1+\dots+p_m=1$, and $\mu=\mu(\rho, \{a_i\},\{p_i\})$ be the pushdown measure for this IFS.
The support of $\mu$ is a subset of $I=[0,a_m]$.

Assume $\rho$ to be algebraic so, as is well known,
$\mu$ is exact-dimensional with dimension~$D$, say. 
Fix $\varepsilon>0$ and $n>1$. Let $D_n$ be the disjoint union of intervals
that form the set 
$$
\left\{\sum_{i=1}^{n} b_i\rho^i:
 b_i\in\{a_1,\dots,a_m\}\right\}.
$$
It was shown in \cite{S} that $\#D_n\gg \rho^{-n}$.
Construct $\ell_j$ as follows: $\ell_j =  \mu(J_j)$, where $J_j \in D_n$ such that
$$
|\mu(J_j)-\rho^{Dn}|<\varepsilon.
$$

Put
$$
\zeta(s,\varepsilon,n) = \sum_{1\le j\le \#D_n : |\mu(J_j)-\rho^{Dn}|<\varepsilon}\ell_j^s.
$$

\begin{thm}[H]
  The sequence $\zeta(s,\varepsilon,n)$ converges as $(\varepsilon,n)\to(0,\infty)$
to a meromorphic  function that 
\begin{itemize}
\item \sl is holomorphic on
$ \{s\in\mathbb C  : \Re s>1\}$;
\item \sl has the set of poles $\{s\in\mathbb C : 0<\Re s<1\}$.
\end{itemize}
\end{thm}

\begin{proof}
Fix $s$ with $\Re s>0, n\ge1$ and assume $k>0$. We observe that
$$
\left|\zeta(s,\varepsilon',n+k)-\zeta(s,\varepsilon,n)\right|
=\left|
\sum_{1\le j\le \#D_{n+k} :  |\mu(J_j)-\rho^{D(n+k)}|<\varepsilon'}\ell_j^s-
\sum_{1\le j\le \#D_n :  |\mu(J_j)-\rho^{Dn}|<\varepsilon}\ell_j^s
\right|
$$
$$
\qquad\quad\qquad\qquad\qquad\qquad\le\left|\sum_{\#D_n+1\le j\le\#D_{n+k}} |\ell_j^s|\right|
\ll \rho^k\cdot \max_{\#D_n+1\le j\le\#D_{n+k}} \ell_j^{\Re s+i\Im s}
$$
$$
\ll \rho^k\cdot\rho^{(n+k)\Re s}
\to0, \quad k\to\infty.
$$
Hence there exists 
$$
\lim_{(\varepsilon,n)\to(0,\infty)}\zeta(s,\varepsilon,n)=:\zeta(s).
$$
It's time to switch to Probability and use
Chebysh\"ev's inequality. Roughly speaking,
for ``most'' $j$ we have $\mu(J_j)\approx \rho^{Dn}$, the mean 
of this distribution, and the aforementioned inequality will
allow me to estimate those indices $j$ that yield values of $\mu(J_j)$
that fall sufficiently ``far'' from the mean value.

The mean value of $\mu(J_j)$ is known to be of exponential order
$\rho^{Dn}$ \cite{S}. Hence by Chebysh\"ev's inequality, 
$$
\sum_{j:|\mu(J_j)-\rho^{Dn}|>\varepsilon} \mu(J_j)   < \varepsilon^2,
$$ 
whence
$$
\sum_{j:|\mu(J_j)-\rho^{Dn}|\le\varepsilon} \ell_j   \ge1-\varepsilon^2.
$$
Since $|\ell_j^s|=\ell_j^{\Re s}$,
$$
\left|\sum_{j: |\mu(J_j)-\rho^{Dn}|<\varepsilon} \ell_j^s\right|=
\left|\sum_{j: |\mu(J_j)-\rho^{Dn}|<\varepsilon} \ell_j\cdot \ell_j^{s-1}\right|\ge
(\min_j\ell_j)^{\Re s-1}\cdot (1-\varepsilon^2).
$$
Thus,
$$
m^{-n(\Re s-1)}(1-\varepsilon)\asymp m^{-n(\Re s-1)}\to\infty\Leftrightarrow\Re s<1,
$$
since $\min \ell_j\asymp m^{-n}$.
Hence the set of poles of $\zeta$ contains the vertical strip $\{s: \Re s\in (0,1)\}$.
On the other hand, if $\Re s>1$, then trivially
$$
\left|\sum_{j: |\mu(J_j)-\rho^{Dn}|<\varepsilon} \ell_j^s\right|\le1
$$
whence $\zeta$ is analytic on $\{s\in\mathbb C : \Re s>1\}$. 
\end{proof}

What happens on the boundary? Let's see. Let $s=1+it$. We have
$$
\left|\sum_{j: |\mu(J_j)-\rho^{Dn}|<\varepsilon} \ell_j^{1+it}\right|=
\left|\sum_{j: |\mu(J_j)-\rho^{Dn}|<\varepsilon}\ell_j\cdot \ell_j^{it}\right|=
\left|\sum_{j: |\mu(J_j)-\rho^{Dn}|<\varepsilon}\ell_j\cdot\exp (it\log \ell_j)\right|
$$
$$
=\left|\sum_{j: |\mu(J_j)-\rho^{Dn}|<\varepsilon}\ell_j\cdot(\cos(b\log \ell_j)+i\sin(b\log \ell_j))\right|:=|F_n(t)|.
$$
Set $$F(t)=\lim_{n\to\infty}F_n(t).$$

\begin{thm}[F]
The following properties are either trivial or can be easily proved:
\medskip\noindent
{\bf 1.} $|F(t)|<1$.
\medskip\noindent
{\bf 2.} $F$ is, generally, aperiodic, so not your run-of-the-mill trig sum.
\medskip\noindent
{\bf 3.} $F\in C^\infty(\mathbb R)$.
\end{thm}

\begin{rmk}
{\bf 1.} Theorem H yields a rather crude dichotomy compared to 
affine cases without overlaps from \cite{P} (see below).
Said that, there are no {\sl natural} affine cases - I guess the set
of numbers whose continued fraction expansion contains only 1s and 2s
is the most natural one. The set of poles here is uniformly discrete,
but their picture looks pretty random. No explicit coordinates of these poles
are known.

\medskip
\noindent {\bf 2.}
The history of the problem can be learned from Mark Pollicott's
slides \cite{P} from his talk at the One World Numeration seminar.
I am grateful to Wolfgang Steiner for organizing this talk.
   
\medskip
\noindent {\bf 3.}
I think it is unlikely that $F$ real analytic. A uniformly convergent sequence of real analytic functions can converge to
any continuous function -- even a sequence of polynomials, according to the Weiersta\ss\ theorem whose proof by
Sergei Bernstein used Chebysh\"ev's inequality and served as an inspiration for my proof of Theorem~H. 
\end{rmk}

\end{document}